\documentclass[10pt,twoside,a4paper,abstract=on]{scrartcl}

\usepackage{authblk}

\usepackage{geometry}
\usepackage{multicol}

\usepackage[utf8]{inputenc}
\usepackage{hyperref}
\usepackage{caption}

\usepackage{todonotes}

\usepackage[toc,title,page]{appendix}

\usepackage{amsmath,amsfonts,amssymb,amsthm}
\usepackage{bbm}

\usepackage{enumitem}

\usepackage{tikz}  

\usepackage{algorithm}
\usepackage{algpseudocode}

\newtheorem{Theorem}{Theorem}[section]

\newtheorem{Proposition}[Theorem]{Proposition}

\newtheorem{Remark}[Theorem]{Remark}

\usepackage[a-1b]{pdfx}

\usepackage{graphicx}
\usepackage{subcaption}

\usepackage[maxbibnames=6]{biblatex}
\definecolor{electricultramarine}{rgb}{0.25, 0.0, 1.0}
\definecolor{ikb}{rgb}{0.0, 0.18, 0.65}
\definecolor{green(colorwheel)(x11green)}{rgb}{0.16, 0.5, 0.0}

\newcommand*{\fzcst}[1]{\relax\ifmmode\text{\textcolor{green(colorwheel)(x11green)}{\sout{\ensuremath{#1}}}}\else\textcolor{green(colorwheel)(x11green)}{\sout{#1}}\fi}

\usepackage{todonotes,varwidth}

\title{
	Dynamical Low-Rank Smoothing}
\author[1]{Youssef Marzouk}
\author[2]{Fabio Nobile}
\author[2]{Fabio Zoccolan}
\affil[1]{Department of Aeronautics and Astronautics, Massachusetts Institute of Technology, Cambridge, MA 02139, USA. email: ymarz@mit.edu}
\affil[2]{Institut de Mathématiques, École Polytechnique Fédérale de Lausanne, 1015 Lausanne, Switzerland. email: fabio.nobile@epfl.ch, fabio.zoccolan@epfl.ch}
\date{}

\begin{document}
   
   \maketitle
	
	\begin{abstract}
		 Computational costs often make smoothing procedures prohibitive for high-dimensional data assimilation problems. To address this challenge, we propose a dynamical low-rank approximation (DLRA) methodology for smoothing concerning frameworks based on stochastic differential equations. We extend the previously developed joint mean-and-covariance optimization (JMCO) filtering setting to derive a reduced-order smoother via the Rauch--Tung--Striebel recursion and establish the corresponding Kalman--Bucy smoothing for affine drift dynamics. The resulting algorithms retain the adaptive nature of DLRA while significantly reducing the computational time and storage of the whole smoothing procedure.
    \end{abstract}
    %\tableofcontents
    
    %\listoftodos

	\section{Introduction}
	\addcontentsline{toc}{section}{Introduction}
	
	Numerical models are indispensable tools for investigating complex physical and engineering systems. Their predictive capabilities, however, are inevitably limited by model imperfections, uncertain parameters, and noisy measurements. Consequently, deterministic simulations alone are often insufficient, motivating probabilistic descriptions of the system dynamics together with the incorporation of observational information. This whole procedure is referred to as \emph{data assimilation}.
	
	Data assimilation combines mathematical models with measurements to improve state estimation and uncertainty quantification. Within the Bayesian framework, the unknown state is modeled as a random variable whose probability distribution is sequentially updated as observations become available. In continuous time, this setting is naturally described through stochastic differential equations (SDEs).
	This data assimilation setting is continuously employed in applications ranging from numerical weather prediction and oceanography \cite{cotter2020data,wang2025nonlinear} to engineering \cite{chen2016filtering,falkena2023bayesian,sztamfater2026sequential}. In large-scale systems, however, the computational cost of propagating probability distributions often becomes prohibitive, motivating the development of reduced-order methodologies.
	
	Dynamical low-rank approximation (DLRA) provides an adaptive reduced-order framework in which the approximation subspace evolves together with the solution, avoiding the limitations of fixed reduced bases. Originally introduced for matrix ordinary differential equations (ODEs) in \cite{koch2007dynamical}, DLRA has subsequently been extended to deterministic \cite{bachmayr2025dynamical,dektor2025interpolatory} and random partial differential equations \cite{musharbash2015error,kazashi2021stability,}. More recently, a mathematical rigorous theory has been extended to stochastic differential equations through the dynamically orthogonal framework \cite{bao2026exponential,kazashi2025dynamical,kazashi2026further}, based on \cite{sapsis2009dynamically}; see also \cite{cao2018stochastic} for another formulation.
	
	The combination of DLRA and data assimilation has only recently begun to be investigated. Reduced-order Kalman-type filters based on DLRA were proposed in \cite{marzouk2026filter,nobile2025dynamicallowrankapproximationskalman,nobile2026dynamicallowrankensemblekalman}, leading to efficient ensemble and particle implementations. Related dynamically orthogonal methodologies have also been explored for oceanographic and chaotic systems, as well as in conjunction with Gaussian mixture methods or other strategies \cite{lu2021bayesian,sapsis2013blending,sondergaard2013dataI,sondergaard2013dataII}. These works demonstrate that DLRA can substantially reduce the computational cost of filtering while preserving its adaptive nature, as it can tailor the surrogate dynamics with the information contained in (continuous- or discrete-time) high-dimensional data, while remaining low-dimensional and, hence, cheap to compute.
	
	The present work extends the DLRA filtering framework of \cite{marzouk2026filter} to the \emph{smoothing} setting. Unlike filtering, which conditions only on observations available up to the current time, smoothing exploits the entire observation window to reconstruct past states, leading to more accurate state estimates and uncertainty quantification \cite{sanz2023inverse,sarkka2023bayesian}. Owing to its backward propagation step, however, smoothing is computationally even more demanding than filtering, making reduced-order formulations, and hence DLRA, particularly valuable for high-dimensional problems.
	
	Building upon the joint mean-and-covariance optimization (JMCO) formulation introduced for filtering, we derive its smoothing counterpart through the classical Rauch--Tung--Striebel (RTS) recursion \cite{rauch1965maximum}. Moreover, for affine drift dynamics, we prove that the resulting reduced-order smoother satisfies the same structural properties as the Kalman--Bucy smoother, thereby extending the advantageous foundations of DLRA from filtering to smoothing.
	
	In the SDE setting, the DLRA approximation can be represented through the dynamically orthogonal decomposition, i.e.\ as the following linear combination of $k$ members:
	\begin{equation}\label{eq: DO}
		X^{\mathrm{DLRA}}_t(\omega) = \sum_{i = 1}^{k} U^{i}_tY^{i}_t(\omega), \quad t \geq 0,
	\end{equation}
	where $\{U^{i}\}_{i=1,\dots,k}$ is the deterministic basis with orthonormal columns satisfying \( U_t \dot{U}_t^\top=0\), while $\{Y^{i}\}_{i=1,\dots,k}$ is the stochastic process, both evolving with \(k\ll d\) degrees of freedom \cite{kazashi2025dynamical}. 
	Alternatively to the two-terms approximation in \eqref{eq: DO}, one can look for a three-component surrogate composed by the mean $m$, a deterministic basis $U$ (with the aforementioned features), and a centered stochastic basis $Y$, i.e.\ $X^{\mathrm{DLRA}}_t(\omega) = m^{\mathrm{DLRA}}_t + \sum_{i = 1}^{k} U^{i}_tY^{i}_t(\omega)$, with appropriate properties. Then, equations for the triplet $(m^{\mathrm{DLRA}}, U, Y)$ would be derived very similarly to the framework in \cite{kazashi2025dynamical}. In that case, mean and covariance of $X^{\mathrm{DLRA}}_t$ are simply given by $m^{\mathrm{DLRA}}_t$ and $C_{t}= U_t^{\top} \mathbb{E}[Y_t Y_t^{\top}]U_t$, respectively, where $\mathbb{E}$ denotes the expectation. In this paper we will consider this $3$-term DLRA initiation, as it allows to have an immediate connection with first and second moments of the full-order solution associated to \eqref{eq: DA SDE}.
		Besides providing a mathematically consistent low-rank approximation, this representation naturally induces, after complete discretization, interacting particle systems, making DLRA particularly suitable for ensemble-based implementations, and hence ensemble Kalman-type filters.  
	
	The DLRA background for filtering is recalled in Section~\ref{sec: setting}, whereas we revise the RTS recursion in \ref{sec: RTS}. In Section~\ref{sec: DLRA-JMCO smoother} we derive the proposed DLRA-JMCO smoother and in Section~\ref{sec: DLRA KBF smoother} we establish its connection with the Kalman--Bucy smoother for affine drift. Finally, in Section \ref{sec: num} numerical examples corroborating our theoretical results will be shown.

	\section{Setting: DLRA and JMCO-Filtering}\label{sec: setting}
	
	In this section, we briefly introduce the theoretical setting where the DLRA-JMCO filter was settled \cite{marzouk2026filter}. We will consider the same framework to derive the sought smoothing relations.
	
	Let $\left( \Omega, \mathcal{F}, \mathbb{P}, (\mathcal{F}_t)_{t \geq 0} \right)$ be a filtered complete probability space with the usual conditions \cite[Remark 6.24]{schilling2021brownian}.
	For $T>0$, we consider the following SDE formulation of a continuous-time data assimilation problem
	\begin{equation}\label{eq: DA SDE}
		\begin{aligned}
			\mathrm{d}X_t &= A(t, X_t) \, \mathrm{d}t + Q^{\frac{1}{2}} \, \mathrm{d}W_t \\
			\mathrm{d}Z_t &= H(t) X_t \, \mathrm{d}t + R^{\frac{1}{2}} \, \mathrm{d}B_t,
		\end{aligned}
	\end{equation}
	where at time $t$ the random vector $X(t)=\left(X_1(t), \ldots, X_d(t)\right)^{\top}$ belongs to $L^2(\Omega, \mathbb{R}^d)$, while $Z(t)=\left(Z_1(t), \ldots, Z_h(t)\right)^{\top}$ belongs to $L^2(\Omega, \mathbb{R}^h)$. Here, 
	we consider a (possibly nonlinear) drift $A : [0, \infty) \times \mathbb{R}^d \to \mathbb{R}^d$ and a linear observation operator $H : [0, \infty) \to \mathbb{R}^{h \times d}$. $W$ and $B$ are real $d$-dimensional and $h$-dimensional $(\mathcal{F}_t)$-Brownian motion, 
	i.e., $W(t) = \left(W_1(t), \ldots , W_d(t) \right)^{\top}$ and $B(t) = \left(B_1(t), \ldots , B_h(t) \right)^{\top}$, respectively. Furthermore, $Q \in \mathbb{R}^{d \times d}$ and $R \in \mathbb{R}^{h \times h}$ are symmetric positive semidefinite matrices, which are the noise
	covariances of the state and the observation, respectively. Moreover, $\{W(t)\}_{t\in [0,T]}$, $\{B(t)\}_{t\in [0,T]}$, and $X_0$ are assumed to be independent. \eqref{eq: DA SDE} is characterized by additive noise SDEs, however, as the filtering equations in \cite{marzouk2026filter}, the following discussion can be enlarged to the case of multiplicative-noise-type dynamics and observation, too.
	In \emph{filtering} and \emph{smoothing}, $X_t \in \mathbb{R}^d$ is the “noisy’’ \emph{state} that we are trying to infer,
	whereas $Z_t \in \mathbb{R}^h$ is the “noisy’’ \emph{observation} or \emph{data}. 
		
	In \cite{marzouk2026filter}, to derive DLRA-type filtering equations we consider a discretization perspective, in the spirit of the first derivation of DLRA for ODEs \cite{koch2007dynamical}. First, we consider a Euler–Maruyama discretization
	of \eqref{eq: DA SDE} with uniform time step, namely for all $n \in \{0, \dots, N\}$ with $\Delta t = \frac{T}{N}$, one has
	\begin{equation}\label{DA SDEs EM}
		\begin{aligned}
			X_{n+1} &= X_n + A(t_n, X_n)\Delta t + Q^{\frac{1}{2}} \Delta W_n \\
			Z_{n+1} &= Z_n + H(t_n) X_{n+1} \Delta t + R^{\frac{1}{2}} \Delta B_n,
		\end{aligned}	
	\end{equation}
	where we consider a uniform time–mesh partition
	$\Delta = \{ t_i,\; i \in \{1,\dots,N\} : t_0 = 0,\; t_N = T,\; \text{and}\;
	t_i - t_{i-1} = \Delta t,\; i \in \{1,\dots,N\} \}$, $\Delta W_n: = W(t_{n+1})-W(t_{n}) \sim \mathcal{N}(0,\Delta t I_{d \times d})$ and $\Delta B_n: = B(t_{n+1})-B(t_{n}) \sim \mathcal{N}(0,\Delta t I_{h \times h})$ are Brownian increments, satisfying $\mathbb{E}[\Delta W_n | \mathcal{F}_{t_n}]=0$,  $\mathbb{E}[\Delta B_n | \mathcal{F}_{t_n}]=0$ for all $n$, and $\mathbb{E}[\Delta W_n \Delta W_m^{\top}] = 0$, $\mathbb{E}[\Delta B_n \Delta B_m^{\top}] = 0$, for all $m \neq n$, and, due to the independence of $B$ and $W$, $\mathbb{E}[\Delta W_n \Delta B_m^{\top}] = 0$ for all $m,n$. Notice that in \eqref{DA SDEs EM} the equation of the observation process is discretized “implicitly’’ in the
	linear drift.
	
	First, we assume that our DLRA $X_n^{\mathrm{DLRA}}$ is defined as $X_n^{\mathrm{DLRA}}=m_n^{\mathrm{DLRA}} + U_n^{\top} Y_n$, where the time-dependent $m_n$, $U_n$, and $Y_n$ are the discrete mean, the discrete deterministic and the stochastic basis, respectively.
	Then, we seek updates $(\Delta m_n^{\mathrm{DLRA}}, \Delta U_n, \Delta Y_n)$ which satisfy
\begin{equation}\label{eq: DLRA updates}
	\begin{aligned} m_{n+1}^{\mathrm{DLRA}} = m_n^{\mathrm{DLRA}} + \Delta m_n^{\mathrm{DLRA}}, \quad
	U_{n+1} = U_n + \Delta U_n, \quad
	Y_{n+1} = Y_n + \Delta Y_n,
	\end{aligned}
\end{equation} 
under the following conditions:
one has $m^{\mathrm{DLRA}}_n= \mathbb{E}[X_{n}^{\mathrm{DLRA}}] \in \mathbb{R}^{d}$ for all $n$, which is a deterministic vector; $U_nU_n^{\top} = I_{k \times k}$, $U_n \Delta U_n^{\top} = 0$; and $Y_n \in L^2(\Omega,  \mathbb{R}^k)$ adapted to the filtration $\mathcal{F}_{t_n}$, $\mathbb{E}[Y_n] =0$, and with linearly independent components $Y_n^1, \dots, Y_n^k$ in $L^2(\Omega,  \mathbb{R})$ for all $n$. We refer to the orthogonality condition $U_n \Delta U_n^{\top} = 0$ as the \emph{discrete gauge condition}, as for the limit $\Delta t \to 0$ one formally retrieves the usual gauge condition \(U_t\dot U_t^\top=0\). Then the DLRA covariance is $C_n^{\mathrm{DLRA}} = U_n^{\top} \mathbb{E}[Y_n Y_n^{\top}]U_n$.
	
	In order to determine the equations that the updates $(\Delta m_n^{\mathrm{DLRA}}, \Delta U_n, \Delta Y_n)$ follow,
   we assume that the DLRA at the point $n+1$, i.e.\ $X_{n+1}^{\mathrm{DLRA}}$, minimizes the following local error between its mean $m_{n+1}^{\mathrm{DLRA}}$ and the one $m$ of the process $X_{n+1}$ solution of \eqref{DA SDEs EM}, with starting point $X_n^{\mathrm{DLRA}}$ and conditioned to the observation $Z_{n+1}$, and their respective covariances $C_{n+1}^{\mathrm{DLRA}}$ and $C$ at each $n$, i.e.\
	\begin{equation}\label{eq: JMCO}
	 \begin{aligned}
		\Big( | m_n^{\mathrm{DLRA}} + \Delta m_n^{\mathrm{DLRA}} - m_{\{X_{n+1} \mid Z_{n+1}, X_n = X_n^{\mathrm{DLRA}}\}}|^2& \\
			+ \| (U_n + \Delta U_n)^{\top}\mathbb{E}[(Y_n + \Delta Y_n)(Y_n + \Delta Y_n)^{\top}] &(U_n + \Delta U_n)- C_{\{X_{n+1} \mid Z_{n+1}, X_n = X_n^{\mathrm{DLRA}}\}} \|_{\mathrm{F}}^2 \Big)^{\frac{1}{2}}.
		\end{aligned}
	\end{equation}
	Basically, we are trying to find the best time-dependent $k$-rank surrogate that at each time $t_n$ approximates the discretized true filtering process minimizing the joint local error of the mean and covariance. We justify this choice by noticing that $X_{n+1}$, with starting point $X_n^{\mathrm{DLRA}}$ and conditioned to $Z_{n+1}$ is Gaussian, and hence, it is completely determined by its mean and covariance.
	
	Via solving a first-order optimality system, we retrieve the following discrete DLRA equations to compute the surrogate at time $t_{n+1}$ given the previous triplet solution $(m_{n}^{\mathrm{DLRA}}, U_n, Y_n)$ at time $t_n$:
	\begin{equation}\label{eq: DLRA-JMCO}
		\begin{aligned}
			m_{n+1}^{\mathrm{DLRA}}
			= &m_n^{\mathrm{DLRA}}
			+ \mathbb{E}[A(t_n, X_n^{\mathrm{DLRA}})]\, \Delta t
			+ C_n^{\mathrm{DLRA}} H(t_n)^\top R^{-1} (Z_{n+1} - Z_{n} - H(t_n) m_n^{\mathrm{DLRA}}\, \Delta t),\\
			U_{n+1}
			= &U_n
			+ C_{Y_n}^{-1} \mathbb{E}[Y_n\, \mathring{A}(t_n, X_n^{\mathrm{DLRA}})^\top]
			\, P_{U_n}^\perp\, \Delta t
			+ C_{Y_n}^{-1} U_n Q\, P_{U_n}^\perp\, \Delta t. \\
			Y_{n+1}
			= &Y_n
			+ U_n \mathring{A}(t_n, X_n^{\mathrm{DLRA}})\, \Delta t
			+ U_n Q^{\frac{1}{2}}\, \Delta W_n \\
			&- C_{Y_n}\, U_n H(t_n)^\top R^{-1} H(t_n) U_n^\top Y_n\, \Delta t
			- C_{Y_n}\, U_n H(t_n)^\top\, R^{-\frac{1}{2}}\, \Delta B_n,
		\end{aligned}
	\end{equation}
	where $\mathring{A}(t_n, X_n^{\mathrm{DLRA}}):= A(t_n, X_n^{\mathrm{DLRA}}) - \mathbb{E}[A(t_n, X_n^{\mathrm{DLRA}})]$
	is the centered drift. Notice that in the case of the Ensemble DLRA-JMCO filter we make evolve $M$ particles of the stochastic basis, computed using $M$ realizations of the initial condition $X_0$, i.e.\ $(X_0(\omega_i)_{i=1,\dots, M}$), and of the Brownian motions $W_t$, i.e.\ $(W_t(\omega_i)_{i=1,\dots, M})$, and $B_t$, i.e.\ $(B_t(\omega_i)_{i=1,\dots, M})$, satisfying
	\begin{equation}\label{eq: ens DLRA Y}
		\begin{aligned}
			Y_{n+1}^{i}
		= &Y_n^{i}
		+ U_n \mathring{A}(t_n, U_n^\top Y_n^{i})\, \Delta t
		+ U_n Q^{\frac{1}{2}}\, \Delta W_n^{i}\\
		&- C_{Y_n}\, U_n H(t_n)^\top R^{-1} H(t_n) U_n^\top Y_n^{i}\, \Delta t
		- C_{Y_n}\, U_n H(t_n)^\top\, R^{-\frac{1}{2}}\, \Delta B_n^{i},
		\end{aligned}
	\end{equation}
	where, with a slight abuse of notation, $C_{Y_n} = \frac{1}{M-1} \sum_{i=1}^{M} Y_n^{i}(Y_n^{i})^{\top}$ is the empirical Gramian of the stochastic basis. Notice that the expectations in \eqref{eq: DLRA-JMCO} are approximated by the empirical averages, too, leading to a noisy interacting particle system.
	\begin{Remark}[DLRA-JMCO filtering in continuous-time]
	In order to retrieve differential equations that our DLRA-JMCO filter has to satisfy in continuous time, we formally consider the limit for the step-size $\Delta t$ going to $0$ in \eqref{eq: DLRA-JMCO}. Then, the triplet $(m^{\mathrm{DLRA}}_t, U_t, Y_t)_{t \geq 0}$ is defined via the following relation
	\begin{align}\label{eq: m_t JMCO}
		\mathrm{d} m_t^{\mathrm{DLRA}}
		=& \mathbb{E}\!\left[ A\!\left(t, X_t^{\mathrm{DLRA}}\right) \right] \mathrm{d}t
		+U_t^{\top} C_{Y_t}U_t\, H(t)^\top R^{-1}
		\bigl( \mathrm{d}Z_t - H(t)\, m_t^{\mathrm{DLRA}}\, \mathrm{d}t \bigr),\\
		\mathrm{d} U_t
		=& C_{Y_t}^{-1}
		\Bigl(
		\mathbb{E}\!\left[ Y_t\, \left(A(t, X_t^{\mathrm{DLRA}})^{\top} 
		- \mathbb{E}\!\left[ A\!\left(t, X_t^{\mathrm{DLRA}}\right)^{\top} \right] \right) \right]
		+ U_t Q
		\Bigr)
		P_{U_t}^{\perp}\, \mathrm{d}t, \label{eq: U_t JMCO}
		\\
		\mathrm{d} Y_t
		=&
		\left(U_t\Bigl(
		A\!\left(t, X_t^{\mathrm{DLRA}}\right)
		- \mathbb{E}\!\left[A\!\left(t, X_t^{\mathrm{DLRA}}\right)\right]
		\Bigr)
		- C_{Y_t}U_t H(t)^\top R^{-1} H(t) U_t^\top Y_t\right) \mathrm{d}t \label{eq: Y_t JMCO}
		\\
		&\qquad
		+ U_t Q^{\frac{1}{2}}\, \mathrm{d}W_t
		- C_{Y_t}U_t  H(t)^\top R^{-\frac{1}{2}}\, \mathrm{d}B_t \nonumber .
	\end{align}	
	Then, one retrieves the DLRA process $X_t^{\mathrm{DLRA}} = m_t^{\mathrm{DLRA}} +U_t^{\top} Y_t$ via Itô's formula \cite{kloeden1992stochastic}.
	\end{Remark}
	
	Equations \eqref{eq: DLRA-JMCO}-\eqref{eq: ens DLRA Y} were the basis to derive the continuous-time equations described in \eqref{eq: m_t JMCO}-\eqref{eq: U_t JMCO}-\eqref{eq: Y_t JMCO} and, hence, might represent a possible discretization. However, some numerical and theoretical issue arise. First, one needs to guarantee orthonormality of the basis $U_n$ at each time, in order to be in compliance with the continuous-time case. Secondly, \eqref{eq: DLRA-JMCO}-\eqref{eq: ens DLRA Y} are not of straightforward interpretation when dealing with assimilation of data that occurs in larger times than that the ones of dynamics computation, hence asking for a more explicit distinction between the prediction and the analysis step. Furthermore, in \cite{kazashi2026dynamicalpartI,kazashi2026dynamicalpartII} staggered procedures, characterized by evolving first the stochastic basis and then the deterministic one, show better stability features, than completing explicit updates, suggesting to use such a procedure in the prediction step. Always for the sake of stability, in the analysis step, one considers staggered methods using the just precomputed deterministic basis and the covariance. Keeping in mind the aforementioned considerations, we state the algorithm in a \emph{prediction}-\emph{analysis} format, i.e.\ via first making evolving the dynamics (i.e.\ the prediction) and then changing the just-obtained measure in based of the acquired data, as in the usual Bayesian setting (i.e.\ the analysis). Then, the whole DLRA-JMCO filter reads as follows: for the prediction update one has
	\begin{equation}
		\begin{aligned}\label{eq: DLRA JMCO prediction}
			\widehat{m}_{n+1}
			=&  m_n^{\mathrm{DLRA}}
			+  \mathbb{E}\!\left[ A\!\left(t_n, X_n^{\mathrm{DLRA}}\right) \right]\Delta t \\
			\widetilde{U}_{n+1}
			=& U_n+ C_{\widetilde{Y}_{n+1}}^{-1}
			\Bigl(
			\mathbb{E}\!\left[ \widetilde{Y}_{n+1}\, \left(A(t_n, X_n^{\mathrm{DLRA}})
			- \mathbb{E}\!\left[ A\!\left(t_n, X_n^{\mathrm{DLRA}}\right) \right]\right)^{\top}\right]
			+ U_n Q
			\Bigr)
			P_{U_n}^{\perp}\, \Delta t
			\\
			\widetilde{Y}_{n+1}
			=&  Y_n
			+ U_n  \left(A(t_n, X_n^{\mathrm{DLRA}})
			- \mathbb{E}\!\left[ A\!\left(t_n, X_n^{\mathrm{DLRA}}\right) \right]\right) \Delta t + U_n Q^{\frac{1}{2}} \Delta W_n,
		\end{aligned}
	\end{equation}
	after which we apply an orthonormalization operation $\widehat{U}_{n+1}^{\top}\widehat{Y}_{n+1}=\widetilde{U}_{n+1}^{\top}\widetilde{Y}_{n+1}$ with $\widehat{U}_{n+1}\widehat{U}_{n+1}^{\top} = I_{k\times k}$, then
	the following analysis step takes place
	\begin{equation}\label{eq: DLRA JMCO analysis}
		\begin{aligned}
			m_{n+1}^{\mathrm{DLRA}}
			= &\widehat{m}_{n+1}
			+ \widehat{U}_{n+1}^{\top}C_{\widehat{Y}_{n+1}}\widehat{U}_{n+1} H(t_n)^\top R^{-1} (Z_{n+1} - Z_n - H(t_n)m_{n+1}^{\mathrm{DLRA}} \Delta t)\\
			U_{n+1}
			= & \widehat{U}_{n+1}\\
			Y_{n+1}
			= &\widehat{Y}_{n+1}
			- C_{\widehat{Y}_{n+1}}\, \widehat{U}_{n+1} H(t_n)^\top R^{-1} H(t_n) \widehat{U}_{n+1}^{\top} Y_{n+1}\, \Delta t
			- C_{\widehat{Y}_{n+1}}\, \widehat{U}_{n+1} H(t_n)^{\top}\, R^{-\frac{1}{2}}\, \Delta B_n,
		\end{aligned}
	\end{equation}
	where we consider a semi-implicit update in this last step for stability purposes.
    Notice that an additional orthonormalization step after the analysis update does not take place as the deterministic basis remains unchanged. 

	\section{DLRA Kalman Smoothing}
	
	The goal of (Bayesian) smoothing is to compute the marginal posterior distribution of the state $X_{k}$ given all the observations acquired up to a certain time $t_{n+1}$, i.e.\ $Z_{n+1}, Z_n, \dots, Z_1$, for all the points of the mesh $k$ such that $1\leq k \leq n+1$. Unlike filtering, the obtained observations can change the distribution of the system at previous times, too. In this work, we are considering a forward-backward smoothing procedure: given $Z_{n+1}$, we update the distribution of $X_{n+1}$ via standard filtering, then we change the distribution of the previous states ``moving backward" through Bayesian updates.
	
	We look for extending the DLRA-JMCO filtering equations in \eqref{eq: DLRA-JMCO} to the smoothing setting. In the discretized setting with affine drift, conditioned on the initial point $X_n^{\mathrm{DLRA}}$ and on the given observation $Z_{n+1}$, our final update is Gaussian, hence we just have to determine the mean and the covariance of our surrogate to completely define our approximation. In the Gaussian case, a possibility to efficiently obtain the equations for the smoothed mean and covariance is through the Rauch--Tung--Striebel (RTS) recursion \cite{rauch1965maximum,sarkka2023bayesian}, which is also called the Kalman smoother. A similar approach is extended in the case of Ensemble RTS methods. Following these observations, we want to use the same formalism to obtain relations for the smoothed triplet $(m_{k}^{\mathrm{DLRA},s}, U_k^s, (Y_k^{s,i})_{i=1,\dots,M})$ of the DLRA-JMCO filter and for the linked Kalman-Bucy procedure for all $k=n+1,n, \dots, 1$, exploiting the low-rank features of these bases. 
	
	Notice that the RTS formalism has been proposed for Gaussian probability distributions defined all over the domain $\mathbb{R}^d$. This is not the case for the DLRA, as the support of our surrogate is defined through the deterministic basis $U_n$. We formally use this recursion to change the distribution living in $U_n$ via considering the pseudoinverse of the minimal norm, which is unique, for the covariance of the DLRA. Indeed, as we will see, the smoothed subspace $U_n^s$ will remain unchanged with respect to the analysis step.
	
	At time $t_n$, via using the superscript $s$ we denote the mean, deterministic and stochastic basis, and covariance obtained by the smoothing procedures, i.e.\ $m_{k}^{\mathrm{DLRA},s}$, $U_{k}^{s}$, $Y_{k}^{s}$, and $C_{k}^{\mathrm{DLRA},s}$, whereas $m_{k}^{\mathrm{DLRA}}$, $U_{k}$, $Y_{k}$, and $C_{k}^{\mathrm{DLRA}}$ indicates the same variables at the an analysis step, and finally $\widehat{m}_{k}$, $\widehat{U}_{k}$, $\widehat{Y}_{k}$, and $\widehat{C}_{k}$, are the predicted mean, bases, and covariance, respectively. 
	
	Before going into further details, we first recall how the RTS recursion works.
	\subsection{Rauch--Tung--Striebel Smoother}\label{sec: RTS}
	
	As previously mentioned, when we are evolving a Gaussian distribution over time, one needs to only express its mean and covariance to characterize the whole measure. 
	The Rauch--Tung--Striebel recursion wants to update all the means $(m_n)_{k \leq n}$ and covariances $(C_n)_{k \leq n}$ in two phases. The former is to evolve \emph{forward} the distribution in a prediction-analysis formalism, like in the Kalman filtering \cite[Chapter 8]{sanz2023inverse}. This update gives the distribution of the random variable $X_{n+1} \mid Z_{n+1}$, via defining the filtered mean $m_{n+1}$ and covariance $C_{n+1}$. But, determining $X_{n+1} \mid Z_{n+1}$ is equivalent to determine the distribution $X_{n+1} \mid Z_{n+1}, \dots, Z_{1}$, hence our smoothed updates at the point $k=n+1$ are the filtered ones, i.e.\ $m_{n+1}^s=m_{n+1}$ and, for the covariance, $C_{n+1}^s=C_{n+1}$. This phase is called the \emph{forward} step.

	The predicted mean and covariance for \eqref{DA SDEs EM} are defined, respectively, by
	\begin{equation}\label{eq: predicted Kalman}
		\widehat{m}_{n+1}
		=m_n+\mathbb{E}[A(t_n,X_n)]\Delta t, \qquad \widehat{C}_{n+1} = \mathbb{E}[(\widehat{X}_{n+1}-\widehat{m}_{n+1})(\widehat{X}_{n+1}-\widehat{m}_{n+1})^{\top}],
	\end{equation}
	where the predicted state is
	\begin{equation*}
	\begin{aligned}
		\widehat{X}_{n+1}
		=X_n+A(t_n,X_n)\Delta t+
		Q^{\frac{1}{2}}\Delta W_n,
	\end{aligned}
	\end{equation*}
	Up to discarding terms of order higher than $O(\Delta t)$, numerically stable Kalman filtered equations are given by 
	\begin{equation}\label{eq: Kalman filter}
	\begin{aligned}
		m_{n+1}:=& \widehat{m}_{n+1}
		+ \widehat{C}_{n+1}\,H(t_n)^{\top}\,R^{-1}
		\left( Z_{n+1} - Z_n - H(t_n)\,  \widehat{m}_{n+1} \Delta t \right) \\
		C_{n+1}
		:= & 
		\widehat{C}_{n+1} - \widehat{C}_{n+1} H(t_n)^{\top}\,R^{-1}H(t_n) \widehat{C}_{n+1} \Delta t
	\end{aligned}
	\end{equation}
	
	To update the distribution of all other $(X_k)_{k \leq n+1}$, given the new observation $Z_{n+1}$, one exploits the Markovianity of the system through Bayesian updates to obtain $X_k^s$ for $k=n,n-1,\dots,0$, given the already updated smoothed distribution of $X_{k+1}^s$. This is the so-called \emph{backward} step.
	
	The smoothing equations of the mean $m_k^s$ and the covariance $C_{k}^s$ associated to the system \eqref{eq: DA SDE} are given as follows \cite[Theorem 8.2]{sarkka2023bayesian}
	
	\begin{Proposition}[Kalman Smoother]\label{prop: Kalman smoother}
		Consider the discretized updates in \eqref{DA SDEs EM}. Then, at time $t_{n+1}$ the associated smoothed mean $m_{n+1}^s$ and covariance $C_{n+1}^s$ satisfy $m_{n+1}^s=m_{n+1}$ and $C_{n+1}^s=m_{n+1}$, where $m_{n+1}$ and $C_{n+1}$ are defined through the analysis step \eqref{eq: Kalman filter}. 
		On the other hand, for all $k \leq n$, one has that 
		\begin{equation}
			\begin{aligned}
				m_{k}^s = m_k + G_k [ m_{k+1}^s - \widehat{m}_{k+1}], \qquad C_{k}^s = C_k + G_k [ C_{k+1}^s - \widehat{C}_{k+1}] G_k^{\top}
			\end{aligned}
		\end{equation}
		where $G_k := \mathbb{E}[X_k (X_k + A(t_k, X_k)^{\top}\Delta t)^{\top}] \widehat{C}_{k+1}^{-1}$, and $\widehat{m}_k$, $\widehat{C}_k$, $m_k$, and $C_k$ are the predicted mean and covariance and the filtered ones at the point $k$ given by \eqref{eq: predicted Kalman}-\eqref{eq: Kalman filter}, respectively.
\end{Proposition}
   
   Proposition \ref{prop: Kalman smoother} provides the equation for the mean and the covariance of our surrogate, and of course they give the exact distribution of the system when we are dealing with an affine drift. Therefore, we will make use of this result to obtain the smoothing updates for the DLRA Kalman-Bucy type filter. On the other hand, to obtain the updates for general $(m_{k}^{\mathrm{DLRA},s}, U_k^s, (Y_k^{s,i})_{i=1,\dots,M})$, as the stochastic discretization of DLRA advances a noisy interacting particle system, we exploit the updates of the Ensemble RTS Smoother \cite{raanes2016ensemble}. This formalism allows us to simplify the derivation of the updates, giving a practical algorithm to implement, which is the ensemble-type generalization of the RTS smoother.

	In this regard, consider to evolve $M$ samples of $X_{n+1}$ over time, i.e.\ $(X_{n+1}^{i})_{i=1,\dots,M}$, which are computed using $M$ realizations of the initial condition $X_0$, i.e.\ $(X_0(\omega_i)_{i=1,\dots, M}$), and of the Brownian motions $W_t$, i.e.\ $(W_t(\omega_i)_{i=1,\dots, M})$, and $B_t$, i.e.\ $(B_t(\omega_i)_{i=1,\dots, M})$. Define the ensemble matrices for the prediction and the analysis state as $\widehat{\mathbb{X}}_{n+1} = \bigl[\widehat{X}_{n+1}^{1}, \dots, \widehat{X}_{n+1}^{M}\bigr] \in \mathbb{R}^{d \times M}$ and $\mathbb{X}_{n+1} = \bigl[X_{n+1}^{1}, \dots, X_{n+1}^{M}\bigr]\in \mathbb{R}^{d \times M}$, respectively. Then, for each $i \in \{1, \dots, M\}$ the equation of $X_{n+1}^{i}$ for the prediction reads as
	\begin{equation*}
		\begin{aligned}
			\widehat{X}_{n+1}^{i}=X_n^{i}+A(t_n,X_n^{i})\Delta t+Q^{\frac{1}{2}}\Delta W_n^{i},
		\end{aligned}
	\end{equation*}
	 and, discarding terms of order higher than $O(\Delta t)$ and preserving numerical stability, for the analysis step one has
	\begin{equation}\label{eq: smoothed X ensemble}
		\begin{aligned}
			X_{n+1}^{i}=\widehat{X}_{n+1}^{i}+\widehat{C}_{n+1}H(t_n)R^{-1}
			\left(Z_{n+1}-Z_n-H(t_n)\widehat{X}_{n+1}^{i}\Delta t+R^{\frac{1}{2}}\Delta B_n^{i},
			\right),
		\end{aligned}
	\end{equation}
	where we defined the predicted ensemble covariance of $(X_{n+1}^{i})_{i=1,\dots,M}$ as $\widehat{C}_{n+1} := \frac{1}{M-1}\widehat{\mathbb{X}}_{n+1}\widehat{\mathbb{X}}_{n+1}^{\top} \in \mathbb{R}^{d \times d}$. Likewise, we define the filtered ensemble covariance as $C_{n+1} := \frac{1}{M-1}\mathbb{X}_{n+1}\mathbb{X}_{n+1}^{\top} \in \mathbb{R}^{d \times d}$, as well as, with a slight abuse of notation, the predicted and filtered ensemble mean as, respectively,
	$$\widehat{m}_{n+1} = \frac{1}{M} \sum_{i=1}^{M} \widehat{X}_{n+1}^{i}, \qquad m_{n+1} = \frac{1}{M} \sum_{i=1}^{M} X_{n+1}^{i}.$$
	
	Then, the proposed smoothed update of the state $(X_{k}^{s,i})_{i=1,\dots,M}$, as well as the mean update $m_k^s$, for all $k \leq n$ are given by the following proposition 
	\begin{Proposition}[Ensemble Kalman Smoother \cite{raanes2016ensemble}]\label{prop: Ens RTS}
	Consider $M$ samples of the fully discretized SDE system in \eqref{DA SDEs EM}. Then, for all $i \in \{1,\dots,M\}$, the smoothed particle system associated to \eqref{DA SDEs EM} satisfies the following relations: at time $t_{n+1}$ one has
	$X_{n+1}^{s,i}=X_{n+1}^{i}$, for $X_{n+1}^{i}$ satisfying \eqref{eq: smoothed X ensemble}, and for all $k=n,n-1,\ldots,1$
	\begin{equation}\label{eq: smooth ens eq}
		 X_k^{s,i}=X_k^{i}+G_k\left(X_{k+1}^{s,i}-\widehat{X}_{k+1}^{i}\right), \qquad m_k^{s}=m_k+G_k\left(m_{k+1}^{s}-\widehat{m}_{k+1}\right),
	\end{equation}
	where we have
	\begin{equation*}
		\begin{aligned}
			\rho_k &= \mathbb{X}_k-m_k \mathbbm{1}^{\top},&	\widehat{\rho}_{k+1}&=\widehat{\mathbb{X}}_{k+1}-\widehat{m}_{k+1}\mathbbm{1}^{\top},\\
			C_{k,k+1}&=\frac{1}{M-1}\rho_k\widehat{\rho}_{k+1}^{\top},&G_k&=C_{k,k+1}\widehat{C}_{k+1}^{\dagger}
		\end{aligned}
	\end{equation*}
   for $\mathbbm{1} = \big[1, \dots, 1\big] \in \mathbb{R}^M$, where $\widehat{X}_k^{i}$, $\widehat{m}_k$, and $\widehat{C}_k$ are the $i$-th particle of the predicted state, the mean, and the ensemble covariance, and $X_k^{i}$, $m_k$, and $C_k$ are the $i$-th particle of the filtered ones, respectively, and $\dagger$ denotes the pseudoinverse of minimal norm.
\end{Proposition}
	
	\subsection{DLRA-JMCO Smoother}\label{sec: DLRA-JMCO smoother}
	
	As we are working with the ensemble DLRA-JMCO filter, we actually consider the set of $M$ particles $(Y_{k}^{s,i})_{i=1,\dots,M}$, $(Y_{k}^{i})_{i=1,\dots,M}$, $(\widehat{Y}_{k}^{i})_{i=1,\dots,M}$, for the smoothed, filtered, and predicted stochastic basis, respectively. Analogously to the notation of \eqref{sec: RTS}, we can collect all the realizations in the following sample matrices:
	\begin{equation*}
		\mathbb{Y}^s_k = \bigl[Y_{k}^{s,1}, \dots, Y_{k}^{s,M} \bigr] \in \mathbb{R}^{k \times M}, \quad \mathbb{Y}_k = \bigl[Y_{k}^{1}, \dots, Y_{k}^{M} \bigr] \in \mathbb{R}^{k \times M}, \quad \widehat{\mathbb{Y}}_k = \bigl[\widehat{Y}_{k}^{1}, \dots, \widehat{Y}_{k}^{M} \bigr] \in \mathbb{R}^{k \times M}.
    \end{equation*}
    We assume from now on that all the deterministic and stochastic bases are of full rank equal to $k$, and hence the covariances of the stochastic components are invertible. However, the proposed equations will remain still valid also when considering their pseudoinverse. 
		
	The Ensemble Rauch-Tung-Striebel recursion explained in Proposition \ref{prop: Ens RTS} allows to recover the state and the covariance of an ensemble Kalman-type filter via closed formulae involving matrix multiplications. Exploiting the orthogonality of the deterministic basis, we can retrieve cheap-to-compute updates for the smoothed triplet ($m_{k}^{\mathrm{DLRA},s}$, $U_{k}^{s}$, $(Y_{k}^{s,i})_{i=1,\dots,M}$), as we see in the following Proposition.
	\begin{Proposition}[Ensemble DLRA-JMCO Smoother]\label{prop: DLR RTS}	
	The DLRA-JMCO Smoother in a forward-backward formalism satisfies the following relations:
	for the forward step at time $t_{n+1}$, the triplet $(m_{n+1}^{\mathrm{DLRA},s}, U_{n+1}^{s},$ $(Y_{n+1}^{s,i})_{i=1,\dots,M})$ satisfies the filtering relations \eqref{eq: DLRA JMCO prediction}-\eqref{eq: DLRA JMCO analysis}, for the mean, the deterministic basis, and the stochastic one, respectively, 
	whereas for the backward step, for all $k=n, n-1, ..., 0$, the triplet $(m_{k}^{\mathrm{DLRA},s}, U_{k}^{s}, (Y_{k}^{s,i})_{i=1,\dots,M})$ satisfies, 
	\begin{equation}\label{eq: DLR-JMCO smoother updates}
		\begin{aligned}
			m_{k}^{\mathrm{DLRA},s}= &m^{\mathrm{DLRA}}_{k}+U_k^{\top}	\mathbb{Y}_k\widehat{\mathbb{Y}}_{k+1}^{\top}
			(\widehat{\mathbb{Y}}_{k+1}\widehat{\mathbb{Y}}_{k+1}^{\top})^{-1} \widehat{U}_{k+1}\big[m^{\mathrm{DLRA},s}_{k+1}-\widehat{m}_{k+1}\big]\\
			U_{k}^{s} = U_{k}=\widehat{U}_{k}, & \quad Y_k^{s,i}
			=
			Y_k^{i}
			+
			\mathbb{Y}_k\widehat{\mathbb{Y}}_{k+1}^{\top}
			(\widehat{\mathbb{Y}}_{k+1}\widehat{\mathbb{Y}}_{k+1}^{\top})^{-1}
			\big[
			Y_{k+1}^{s,i}
			-
			\widehat{Y}_{k+1}^{i}
			\big].
		\end{aligned}
	\end{equation}
	Finally, one has  $X_{k}^{\mathrm{DLRA},s,i} = m_{k}^{\mathrm{DLRA},s} +(U_{k}^{s})^{\top}Y_{k}^{s,i}$.
	\begin{proof}
	The proof follows by employing relations \eqref{eq: smooth ens eq} for the DLRA equations. Using the same notation as in Proposition \ref{prop: Ens RTS}, we have
	\begin{equation*}
		\begin{aligned}
			\rho_k
			=
			U_k^{\top}\mathbb{Y}_k,
			\qquad
			\widehat{\rho}_{k+1}
			=
			\widehat{U}_{k+1}^{\top}\widehat{\mathbb{Y}}_{k+1},
		\end{aligned}
	\end{equation*}
	as the stochastic basis updates are (formally) assumed to be centered, and 
	\begin{equation*}
		\begin{aligned}
			C_{k,k+1}
			=\frac{1}{M-1}U_k^{\top}\mathbb{Y}_k\widehat{\mathbb{Y}}_{k+1}^{\top}\widehat{U}_{k+1},
			\qquad
			\widehat{C}_{k+1}=\frac{1}{M-1}\widehat{U}_{k+1}^{\top}\widehat{\mathbb{Y}}_{k+1}\widehat{\mathbb{Y}}_{k+1}^{\top}\widehat{U}_{k+1}.
		\end{aligned}
	\end{equation*}
	As $\left( \widehat{U}_{k+1}^{\top}(\frac{1}{M-1}\widehat{\mathbb{Y}}_{k+1}\widehat{\mathbb{Y}}_{k+1}^{\top}) \widehat{U}_{k+1}\right)^{\dagger} =\widehat{U}_{k+1}^{\top}(\frac{1}{M-1}\widehat{\mathbb{Y}}_{k+1}\widehat{\mathbb{Y}}_{k+1}^{\top})^{-1} \widehat{U}_{k+1}$ thanks to the orthonormality of the rows of $\widehat{U}_{k+1}$, then the mixed covariance update $G_k$ reads as 
\begin{equation*}
	\begin{aligned}
		G_k&=\frac{1}{M-1}U_k^{\top}\mathbb{Y}_k\widehat{\mathbb{Y}}_{k+1}^{\top}\widehat{U}_{k+1}\widehat{U}_{k+1}^{\top} (M-1)(\widehat{\mathbb{Y}}_{k+1}\widehat{\mathbb{Y}}_{k+1}^{\top})^{-1}\widehat{U}_{k+1},
		\\
		&=
		U_k^{\top}
		\mathbb{Y}_k\widehat{\mathbb{Y}}_{k+1}^{\top}
		(\widehat{\mathbb{Y}}_{k+1}\widehat{\mathbb{Y}}_{k+1}^{\top})^{-1}
		\widehat{U}_{k+1},
	\end{aligned}
\end{equation*}
where we exploit the orthogonality of the rows of $\widehat{U}_{k+1}$. 
Concerning the mean, one has
\begin{equation*}
	\begin{aligned}
		m_n^{\mathrm{DLRA},s}
		=
		m_n^{\mathrm{DLRA}}
		+
		U_n^{\top}
		\mathbb{Y}_n\widehat{\mathbb{Y}}_{n+1}^{\top}
		(\widehat{\mathbb{Y}}_{n+1}\widehat{\mathbb{Y}}_{n+1}^{\top})^{-1}
		\widehat{U}_{n+1}
		\left(
		m_{n+1}^{\mathrm{DLRA},s}
		-
		\widehat{m}_{n+1}
		\right),
	\end{aligned}
\end{equation*}

Now we want to derive updates for the deterministic and stochastic bases. Via using relation  \eqref{eq: smooth ens eq}, one has that 
\begin{equation}\label{eq: intermediate X}
	\begin{aligned}
		X_n^{\mathrm{DLRA},s,i} = & m_n^{\mathrm{DLRA},s} +(U_n^s)^{\top}Y_n^{s,i} \\
		=&\, m_n^{\mathrm{DLRA},s} +
		U_n^{\top}Y_n^{i}
		+
		U_n^{\top}
		\mathbb{Y}_n\widehat{\mathbb{Y}}_{n+1}^{\top}
		(\widehat{\mathbb{Y}}_{n+1}\widehat{\mathbb{Y}}_{n+1}^{\top})^{-1}
		\widehat{U}_{n+1}
		\left(
		(U_{n+1}^{s})^{\top}Y_{n+1}^{s,i}
		-
		\widehat{U}_{n+1}^{\top}\widehat{Y}_{n+1}^{i}
		\right)
		\\
		=&\, m_n^{\mathrm{DLRA},s} +
		U_n^{\top}
		\Big(
		Y_n^{i}
		+
		\mathbb{Y}_n\widehat{\mathbb{Y}}_{n+1}^{\top}
		(\widehat{\mathbb{Y}}_{n+1}\widehat{\mathbb{Y}}_{n+1}^{\top})^{-1}
		\widehat{U}_{n+1}
		\big[
		(U_{n+1}^{s})^{\top}Y_{n+1}^{s,i}
		-
		\widehat{U}_{n+1}^{\top}\widehat{Y}_{n+1}^{i}
		\big]
		\Big).
	\end{aligned}
\end{equation}

From relation \eqref{eq: intermediate X} one can see that the DLRA smoothing procedure
does not change the subspace $U$. Thus $U_{n+1}^{s} = U_{n+1} = \widehat{U}_{n+1}$ for all $n$, and hence from \eqref{eq: intermediate X} and using the fact that $U_k^s$ is always of full rank $k$ we have
\begin{equation*}
	\begin{aligned}
	(U_n^s)^{\top}Y_n^{s,i} 	=
		U_n^{\top}
		\Big(
		Y_n^{i}
		+
		\mathbb{Y}_n\widehat{\mathbb{Y}}_{n+1}^{\top}
		(\widehat{\mathbb{Y}}_{n+1}\widehat{\mathbb{Y}}_{n+1}^{\top})^{-1}
		\big[
		Y_{n+1}^{s,i}
		-
		\widehat{Y}_{n+1}^{i}
		\big]
		\Big).
	\end{aligned}
\end{equation*}
Therefore, for all $k=n,n-1,\ldots,1$ our smoothed stochastic basis evolves by the following relation
\begin{equation*}
	\begin{aligned}
		Y_k^{s,i}
		=
		\Big(
		Y_k^{i}
		+
		\mathbb{Y}_k\widehat{\mathbb{Y}}_{k+1}^{\top}
		(\widehat{\mathbb{Y}}_{k+1}\widehat{\mathbb{Y}}_{k+1}^{\top})^{-1}
		\big[
		Y_{k+1}^{s,i}
		-
		\widehat{Y}_{k+1}^{i}
		\big]
		\Big).
	\end{aligned}
\end{equation*}
\end{proof}
\end{Proposition}

One can notice that \eqref{eq: DLR-JMCO smoother updates} results in cheap-to-compute updates when $k \ll d$ and when one can approximate the (possibly non-linear) drift $A$ in a low-rank format. Furthermore, for stability reasons, we computed the filtered equations for the mean and the stochastic basis in a staggered way (see \cite{marzouk2026filter} for more details). In detail, the smoothing procedure generally needs to store three quantities: the predicted, the filtered, and the smoothed ones. Therefore, DLRA would also enhance savings in terms of storage space.
We summarize the whole procedure in Algorithm \ref{alg: DLRA-JMCO Smoother algorithm}.
	\begin{algorithm}[!h]
	\caption{DLRA-JMCO Smoother}\label{alg: DLRA-JMCO Smoother algorithm}
	\begin{flushleft}
		\textbf{Input}: Number of samples $M$, initial data  ($m_0$, $U_0$, $(Y_0^{i})_{i=1,\dots,M}$), $\Delta t$, Observations $\{Z_n\}_{n\in0,\dots,N}$
		
		\textbf{Output:} smoothed approximation $\{X^{\mathrm{DLRA},s,i}_n= m_n^{\mathrm{DLRA},s} + (U_n^s)^{\top}Y^{s,i}_n\}_{n=0,\ldots, N}$ for $i=1,\dots,M$. 
	\end{flushleft}
	\begin{algorithmic}[1]
		
		\ForAll {$n \in  \{0, \ldots, N-1\}$} 
		
		\State Generate $M$ Brownian increments $\Delta W_n^{i} \sim \mathcal{N}(0,\Delta t I_{m \times m})$, $\Delta B_n^{i} \sim \mathcal{N}(0,\Delta t I_{h \times h})$  with $i \in \{1,\dots,M\}.$
		
		\State Compute $\widehat{m}_{n+1}
		= m_n^{\mathrm{DLRA}}
		+ \frac{1}{M} \sum_{i = 1}^{M} A(t_n,  m_n^{\mathrm{DLRA}} + U_n^{\top}Y_n^{i})\, \Delta t.$
		
		\State Compute $\widetilde{Y}^{i}_{n+1} = Y_n^{i} + U_n \mathring{A}\!\left(t_n, m_n^{\mathrm{DLRA}} + U_n^{\top}Y_n^{i}\right) \Delta t_n +  U_n Q^{\frac{1}{2}} \Delta W_n^{i}$, for all $i \in \{1,\dots,M\},$
		\qquad where $\mathring{A}\!\left(t_n, m_n^{\mathrm{DLRA}} + U_n^{\top}Y_n^{i}\right) = A \!\left(t_n, m_n^{\mathrm{DLRA}}+ U_n^{\top}Y_n^{i}\right) - \frac{1}{M} \sum_{i = 1}^{M} A(t_n,  m_n^{\mathrm{DLRA}} + U_n^{\top}Y_n^{i})$
		
		\State Assemble $\overline{C}_{\widetilde{Y}_{n+1}} = \frac{1}{M-1} \sum_{i = 1}^{M} \widetilde{Y}_{n+1}^{i}(\widetilde{Y}_{n+1}^{i})^{\top}$
		\State Compute $\widetilde{U}_{n+1}$: 
		$$\overline{C}_{\widetilde{Y}_{n+1}} \widetilde{U}_{n+1}
		= \overline{C}_{\widetilde{Y}_{n+1}}  U_n
		+  \frac{1}{M-1} \sum_{i = 1}^{M} \left( \widetilde{Y}_{n+1}^{i}\, \mathring{A}(t_n,  m_n^{\mathrm{DLRA}} + U_n^{\top}Y_n^{i})^\top\right)
		\, P_{U_n}^\perp\, \Delta t
		+ U_n Q\, P_{U_n}^\perp\, \Delta t. $$
		\State Reorthonormalize the deterministic basis: find   $(U_{n+1}, (Y^{i}_{n+1})_{i=1,\dots,M})$ such that:
		\begin{equation*}
			\widehat{U}^{\top}_{n+1} \widehat{Y}_{n+1}^{i}=\widetilde{U}_{n+1}^{\top} \widetilde{Y}_{n+1}^{i}, \quad \widehat{U}_{n+1}\widehat{U}^{\top}_{n+1} = I_{k\times k}.
		\end{equation*}
		via $(\widehat{U}^{\top}_{n+1}, R_{n+1}) = \texttt{QR}(\widetilde{U}^{\top}_{n+1})$ and $\widehat{Y}_{n+1}^{i} = R_{n+1} \widetilde{Y}^{i}_{n+1}$, for all $i \in \{1,\dots,M\}.$
		
		\State Recentering: $\widehat{Y}_{n+1}^{i} = \widehat{Y}_{n+1}^{i} - \frac{1}{M} \sum_{j= 1}^{M}  \widehat{Y}_{n+1}^{j}$  
		for all $i$ and $\widehat{m}_{n+1} = \widehat{m}_{n+1} + \widehat{U}_{n+1}^{\top} \left(\frac{1}{M} \sum_{j= 1}^{M}   \widehat{Y}_{n+1}^{j}\right)$
		
		\State With $\Delta Z_n = Z_{n+1}-Z_n$ solve for $m_{n+1}^{\mathrm{DLRA}}$ and $Y_{n+1}^{i}$
		\begin{equation*}
			\begin{aligned}
				(I_{d \times d}+ \widehat{U}_{n+1}^{\top} \overline{C}_{\widehat{Y}_{n+1}} \widehat{U}_{n+1} H(t_n)^\top R^{-1} H(t_n) \Delta t) m_{n+1}^{\mathrm{DLRA}}
				= &\widehat{m}_{n+1}
				+ \widehat{U}_{n+1}^{\top} \overline{C}_{\widehat{Y}_{n+1}} \widehat{U}_{n+1} H(t_n)^\top R^{-1}  \Delta Z_n,\\
				(I_{k \times k} + \overline{C}_{\widehat{Y}_{n+1}} \, \widehat{U}_{n+1} H(t_n)^\top R^{-1} H(t_n) \widehat{U}^{\top}_{n+1} \Delta t )Y_{n+1}^{i}
				= &\widehat{Y}_{n+1}^{i}
				- \overline{C}_{\widehat{Y}_{n+1}} \, \widehat{U}_{n+1}H(t_n)^{\top} R^{-\frac{1}{2}}\, \Delta B_n^{i}, 
			\end{aligned}
		\end{equation*}
		\quad for all $i \in \{1,\dots,M\},$ respectively, and then set $m_{n+1}^{\mathrm{DLRA},s}=m_{n+1}^{\mathrm{DLRA}}$, $Y_{n+1}^{s,i}=Y_{n+1}^{i}$, $U^s_{n+1} =  U_{n+1} = \widehat{U}_{n+1}$, after another recentering.
		\State 	For all $k=n, n-1, ..., 0$, for all $i \in \{1,\dots,M\},$ with $\Delta Z_k = Z_{k+1}-Z_k$ compute
		\begin{equation*}
			\begin{aligned}
				m_{k}^{\mathrm{DLRA},s}= &m^{\mathrm{DLRA}}_{k}+U_k^{\top}	\mathbb{Y}_k\widehat{\mathbb{Y}}_{k+1}^{\top}
				(\widehat{\mathbb{Y}}_{k+1}\widehat{\mathbb{Y}}_{k+1}^{\top})^{-1} \widehat{U}_{k+1}\big[m^{\mathrm{DLRA},s}_{k+1}-\widehat{m}_{k+1}\big] \\
				U_{k}^{s} = U_{k}, & \quad Y_k^{s,i}
				=
				Y_k^{i}
				+
				\mathbb{Y}_k\widehat{\mathbb{Y}}_{k+1}^{\top}
				(\widehat{\mathbb{Y}}_{k+1}\widehat{\mathbb{Y}}_{k+1}^{\top})^{-1}
				\big[
				Y_{k+1}^{s,i}
				-
				\widehat{Y}_{k+1}^{i}
				\big].
			\end{aligned}
		\end{equation*}
		\EndFor
	\end{algorithmic}
\end{algorithm}

\section{DLRA-JMCO Kalman-Bucy-type Smoothing}\label{sec: DLRA KBF smoother}

We now specialize the DLRA Smoothing equations derived in Proposition \ref{prop: DLR RTS} for the DLRA-JMCO procedure, to the affine-drift case, i.e.\
\begin{equation*}
	A(t,x)=A(t)x+f(t),
\end{equation*}
with $A: [0, \infty) \to \mathbb{R}^{d \times d}$ and $f: [0, \infty) \to \mathbb{R}^{d}$.
In the case of linear dynamics and observations, and Gaussian prior,
one just needs to track the mean and the covariance of the algorithm. In that case, one can obtain respectively a Kalman-Bucy (KB) smoother, as stated in Proposition \ref{prop: Kalman smoother}.

We recall that for the DLRA-JMCO-KB type filter obtained in \cite{marzouk2026filter}, the mean $m_n^{\mathrm{DLRA}}$, the covariance of the stochastic basis $C_{Y_n}$, and the deterministic basis $U_n$ satisfy the following relations
for the prediction step
\begin{equation}\label{eq: DLR JMCO KBF prediction}
	\begin{aligned}
		m_{n+1}
		=&
		m_n^{\mathrm{DLRA}}
		+
		(A(t_n)m_n^{\mathrm{DLRA}} + f(t_n))\Delta t\\
		\widetilde{C}_{\widehat{Y}_{n+1}}
		=& C_{Y_{n}} +
		\bigl(
		U_n A(t_n) U_n^\top C_{Y_{n}}
		+ C_{Y_n} U_n A(t_n)^\top U_n^\top
		+ U_n Q U_n^\top
		\bigr)\, \Delta t \\
		\widetilde{U}_{n+1}
		= &U_n
		+ U_n A(t_n)^\top
		\, P_{U_n}^\perp\, \Delta t
		+ C_{Y_n}^{-1} U_n Q\, P_{U_n}^\perp\, \Delta t, \\
	\end{aligned}
\end{equation}
respectively. Then, to remain in compliance with the continuous-time DLRA, one applies the usual orthonormalization procedure on the deterministic basis, for instance via QR decomposition as $(\widehat{U}_{n+1}^{\top}, R) =\texttt{QR}(\widetilde{U}_{n+1}^{\top})$, and hence we set $C_{\widehat{Y}_{n+1}}= R \widetilde{C}_{\widehat{Y}_{n+1}} R^{\top}$. Then the analysis step reads as
\begin{equation}\label{eq: DLR JMCO KBF analysis}
	\begin{aligned}
		m_{n+1}^{\mathrm{DLRA}}
		=&
		\widehat{m}_{n+1}+
		\widehat{U}_{n+1}^{\top} C_{\widehat{Y}_{n+1}} \widehat{U}_{n+1} H(t_n)^{\top}R^{-1}
		\left(
		Z_{n+1} - Z_n
		-
		H(t_n)m_{n+1}^{\mathrm{DLRA}}\Delta t
		\right)\\
		C_{Y_{n+1}}
		=& 	C_{\widehat{Y}_{n+1}}
		- C_{\widehat{Y}_{n+1}} \widehat{U}_{n+1}  H(t_n)^\top R^{-1} H(t_n)\widehat{U}_{n+1}^\top C_{\widehat{Y}_{n+1}} \Delta t \\
	\end{aligned}
\end{equation}
where the predicted deterministic basis is equivalent to the one of the analysis, i.e.\ $U_{n+1}=\widehat{U}_{n+1}$, and we update the mean and the stochastic basis implicitly for stability reasons.
The following proposition stated the smoothed equations for the mean, the reduced covariance, and the subspace, where we recall that the equation of the predicted, the filtered, and the smoothed covariance of the DLRA are retrieved, respectively, by the following relations for all $k=0,\dots, N$
\begin{equation}\label{eq: covariances}
	\widehat{C}_{k}=\widehat{U}_{k}^{\top}C_{\widehat{Y}_{k}}\widehat{U}_{k}, \qquad C^{\mathrm{DLRA}}_{k}=U_{k}^{\top}C_{Y_{k}}U_{k}, \qquad C^{\mathrm{DLRA},s}_{k}=(U_{k}^s)^{\top}C_{Y_{k}^s}U_{k}^s.
\end{equation}

\begin{Proposition}[DLR JMCO KB Smoother]\label{prop: DLR JMCO KB Smoother}
The proposed equations for DLR JMCO Smoother in a forward-backward formalism in case of linear drift are the following:
for the forward step at time $t_{n+1}$, the triplet $(m_{n+1}^{\mathrm{DLRA},s}, U_{n+1}^{s}, C_{Y_{n+1}}^{s})$ satisfies $m_{n+1}^{s}:=m_{n+1}$, $U_{n+1}^{s}:=U_{n+1}$, $C_{Y_{n+1}}^{s}:=C_{Y_{n+1}}$, i.e.\ the filtering relations \eqref{eq: DLR JMCO KBF analysis},
while for the backward step, for all $k=n, n-1, ..., 0$, one has the following relations
\begin{equation}\label{eq: DLR KBF Smoother}
	\begin{aligned}
	m_k^{\mathrm{DLRA},s}
	=&
	m_k^{\mathrm{DLRA}}
	+
	C_k^{\mathrm{DLRA},s}
	(I_{d \times d}+A(t_k)\Delta t)^{\top}
	U_{k+1}^{\top}
	C_{\widehat{Y}_{k+1}}^{-1}U_{k+1}
	\left(
	m_{k+1}^{\mathrm{DLRA},s}
	-
	\widehat{m}_{k+1}
	\right), \quad U_k^s = U_k,
	\\
		C_{Y_k^{s}}
	=&\,
	C_{Y_k}
	+
	C_{Y_k}
	U_k
	(I_{d \times d}+A(t_k)\Delta t)^{\top}
	U_{k+1}^{\top}
	C_{\widehat{Y}_{k+1}}^{-1}
	\\
	&\cdot
	(C_{Y_{k+1}^{s}}-C_{\widehat{Y}_{k+1}})
	C_{\widehat{Y}_{k+1}}^{-1}
	U_{k+1}
	(I_{d \times d}+A(t_k)\Delta t)
	U_k^{\top}
	C_{Y_k}.
	\end{aligned}
\end{equation}
\begin{proof}
As we are dealing with a linear drift, the evolution of system \eqref{DA SDEs EM} conditioned on the initial point $X_n^{\mathrm{DLRA}}$ and the observation $Z_{n+1}$ is Gaussian. Therefore, to retrieve the sought relations we exploit Proposition \ref{prop: Kalman smoother}, using the same notation employed therein.

As defined in Proposition \ref{prop: Kalman smoother}, for all $k \leq n$ one has
\begin{equation*}
	\begin{aligned}
		G_k
		=
		C_k^{\mathrm{DLRA},s}
		(I_{d \times d}+A(t_k)\Delta t)^{\top}
		\widehat{C}_{k+1}^{\dagger},
	\end{aligned}
\end{equation*}
and the smoothing updates on the DLRA mean and covariance are
\begin{equation}\label{eq: dlr smooth intermediate}
	\begin{aligned}
		m_k^{\mathrm{DLRA},s}
		&=
		m_k^{\mathrm{DLRA}}
		+
		G_k
		\left(
		m_{k+1}^{\mathrm{DLRA},s}
		-
		\widehat{m}_{k+1}
		\right),
		\\
		C_k^{\mathrm{DLRA},s}
		&=
		C_k^{\mathrm{DLRA}}
		+
		G_k
		\left(
		C_{k+1}^{\mathrm{DLRA},s}
		-
		\widehat{C}_{k+1}
		\right)
		G_k^{\top}
		\\
		&=
		C_k^{\mathrm{DLRA}}
		+
		C_k^{\mathrm{DLRA}}
		(I_{d \times d}+A(t_k)\Delta t)^{\top}
		\widehat{C}_{k+1}^{\dagger}
		\left(
		C_{k+1}^{\mathrm{DLRA},s}
		-
		\widehat{C}_{k+1}
		\right)
		\widehat{C}_{k+1}^{\dagger}
		(I_{d \times d}+A(t_k)\Delta t)
		C_k^{\mathrm{DLRA},s}.
	\end{aligned}
\end{equation}
Via using \eqref{eq: covariances} the relation of the covariance can be written as 
\begin{equation}
	(U_{k}^s)^{\top}C_{Y_{k}^s}U_{k}^s = 	U_{k}^{\top} C_{Y_{k}^s}U_{k}\left( I_{d \times d}
	+
	(I_{d \times d}+A(t_k)\Delta t)^{\top}
	\widehat{C}_{k+1}^{\dagger}
	\left(
	C_{k+1}^{\mathrm{DLRA},s}
	-
	\widehat{C}_{k+1}
	\right)
	\widehat{C}_{k+1}^{\dagger}
	(I_{d \times d}+A(t_k)\Delta t)
	C_k^{\mathrm{DLRA},s}\right).
\end{equation}
From the previous relation, we notice that the deterministic subspace $U$ does not change, and, hence, 
\begin{equation}\label{eq: U KBF smoother}
	\begin{aligned}
		U_{k}^{s}
		=
		U_{k}
		=
		\widehat{U}_{k}.
	\end{aligned}
\end{equation}
Using again \eqref{eq: covariances}, one can simplify \eqref{eq: dlr smooth intermediate} to
\begin{equation*}
	\begin{aligned}
		C_k^{\mathrm{DLRA},s}
		=
		C_k^{\mathrm{DLRA}}
		+
		C_k^{\mathrm{DLRA}}
		(I_{d \times d}+A(t_k)\Delta t)^{\top}
		\widehat{U}_{k+1}^{\top}
		C_{\widehat{Y}_{k+1}}^{-1}
		(C_{Y_{k+1}^{s}}-C_{\widehat{Y}_{k+1}})
		C_{\widehat{Y}_{k+1}}^{-1}
		\widehat{U}_{k+1}
		(I_{d \times d}+A(t_k)\Delta t)
		C_k^{\mathrm{DLRA}}.
	\end{aligned}
\end{equation*}

But, using \eqref{eq: U KBF smoother} again, finally one gets
\begin{equation*}
	\begin{aligned}
		C_{Y_k^{s}}
		=&\,
		C_{Y_k}
		+
		C_{Y_k}
		U_k
		(I_{d \times d}+A(t_k)\Delta t)^{\top}
		U_{k+1}^{\top}
		C_{\widehat{Y}_{k+1}}^{-1}
		\\
		&\cdot
		(C_{Y_{k+1}^{s}}-C_{\widehat{Y}_{k+1}})
		C_{\widehat{Y}_{k+1}}^{-1}
		U_{k+1}
		(I_{d \times d}+A(t_k)\Delta t)
		U_k^{\top}
		C_{Y_k}.
	\end{aligned}
\end{equation*}
\end{proof}
\end{Proposition}

\section{Numerical examples: Stochastic Advection-Diffusion-Reaction Model}\label{sec: num}
In this section, we illustrate the performance of the proposed smoothing strategy for a noisy one-dimensional advection-diffusion-reaction system, similarly studied in \cite{marzouk2026filter}, whose diffusion is given by a low-rank additive noise. This problem is described by the following relation:
\begin{equation}\label{ex: SADR DA}
	\begin{cases}
		\begin{aligned}
			\partial_t u(x,t,\omega)&= L u(x,t,\omega) + \sum_{i=1}^{m} \phi_i(x) \dot{W}_i(\omega), \quad (x,t,\omega) \in [0,L] \times [0,T] \times \Omega, \\
			u_{0}(x,\omega)&=  \sum_{i = 1}^m \cos\left( \frac{\pi i x}{ L} \right) \mathcal{N}_{i}(0,1), \quad (x,\omega) \in [0,L] \times \Omega,\\
			\partial_x u(0,t, \omega)&= \partial_x u(1,t,\omega) = 0, \quad (t,\omega) \in [0,T] \times \Omega,
		\end{aligned}
	\end{cases}
\end{equation}
where in \eqref{ex: SADR DA} the linear operator $L$ is defined as $L u(x,t,\omega) = a \partial_x^2 u(x,t,\omega) - v \partial_x u(x,t,\omega) + r u(x,t,\omega),$and $\{\mathcal{N}_{i}(0,1)\}_{i=1,\dots,m}$ denotes a collection of i.i.d.\ standard normal random variables, independent of the one-dimensional Brownian motions $\{W_i\}_{i=1,\dots,m}$. We set $a=0.01$, $v=0.05$, and $r=0.02$. The spatial domain is chosen as $[0,L]$, with $L=5$, and we impose Neumann boundary conditions. The temporal domain is given by $[0,T]$, with $T=20$.

For the diffusion term, we choose

$$
\phi_i(x) = \cos\left(\frac{i\pi f x}{L}\right),
\qquad i=1,\dots,m,
$$

with $f=5$. We discretize \eqref{ex: SADR DA} in space using second-order centered finite differences, while the advection term is approximated by a first-order upwind scheme. We employ a uniform spatial grid with mesh size $\Delta x=0.1$. This results in a physical dimension $d=50$, and we denote by $U^{\Delta x}_t$ the flattened vector containing the spatially discretized components of $u$ at time $t$.

For the data assimilation procedure, we will consider as reference algorithm the \emph{continuous-time ensemble Kalman smoother} in a RTS format \cite{raanes2016ensemble} (ES SDE), which will be computed with a larger number of particles than the DLRA ones, namely $M=10^5$ sample paths. 

The reference true solution of \eqref{ex: SADR DA} is discretized in time using the forward Euler--Maruyama method with uniform time step $\Delta t=1\cdot 10^{-2}$. We take $m=12$ and employ a DLR approximation of rank $k=12$.

We consider the following observation process:
\begin{equation}\label{eq: SADR obs process 1}
	\mathrm{d} Z_{t} = H U^{\Delta x}_t \mathrm{d}t + R^{\frac{1}{2}} \mathrm{d}B_t.
\end{equation}
We observe $10$ equidistant spatial coordinates, so that the dimension $h$ of the observation process $Z_t$ is $h=10$. The covariance of the observation noise is chosen as $R=0.01 \cdot I_{h\times h}.$
Finally, when computing the time-averaged errors, we discard an initial warm-up interval and set the warm-up time to $T_0=2$.

We try to compare the performance of the DLR filters concerning errors. As we are considering a linear SDE problem with a Gaussian initial condition, we know that the measure induced by the problem is still Gaussian and that an ensemble Kalman smoother is a good approximation of \eqref{ex: SADR DA} for large $M$. As the system is low-dimensional, we are expecting a good approximation of the DLRA-JMCO smoother and by construction we expect that this strategy outperforms the related DLRA-JMCO filter, as the smoother conditioned all the past mean and covariances on all the observations.

Indeed, in Figures \ref{fig: SADR CDCO ES SDE rank} and \ref{fig: SADR CDCO ES SDE samples}, where we plot the (relative) errors of the mean and the covariance with respect to the reference ES SDE at final time $T=20$ varying the rank and the number of samples, we see that the errors originated by the DLRA-JMCO Smoother are indeed lower than the respective filter, as expected.

\begin{figure}[!h]
	\centering
	\includegraphics[scale=0.18]{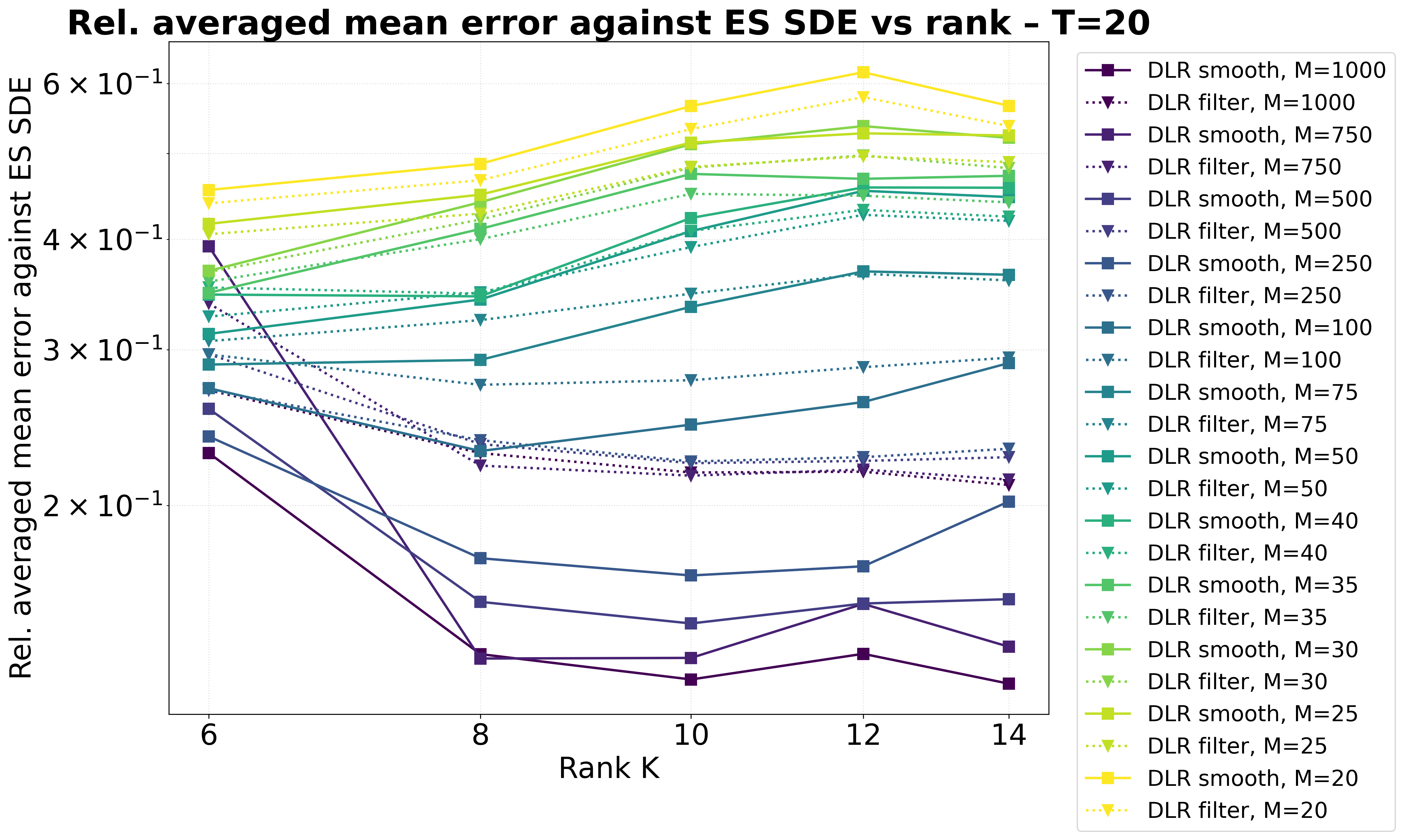}
	\includegraphics[scale=0.18]{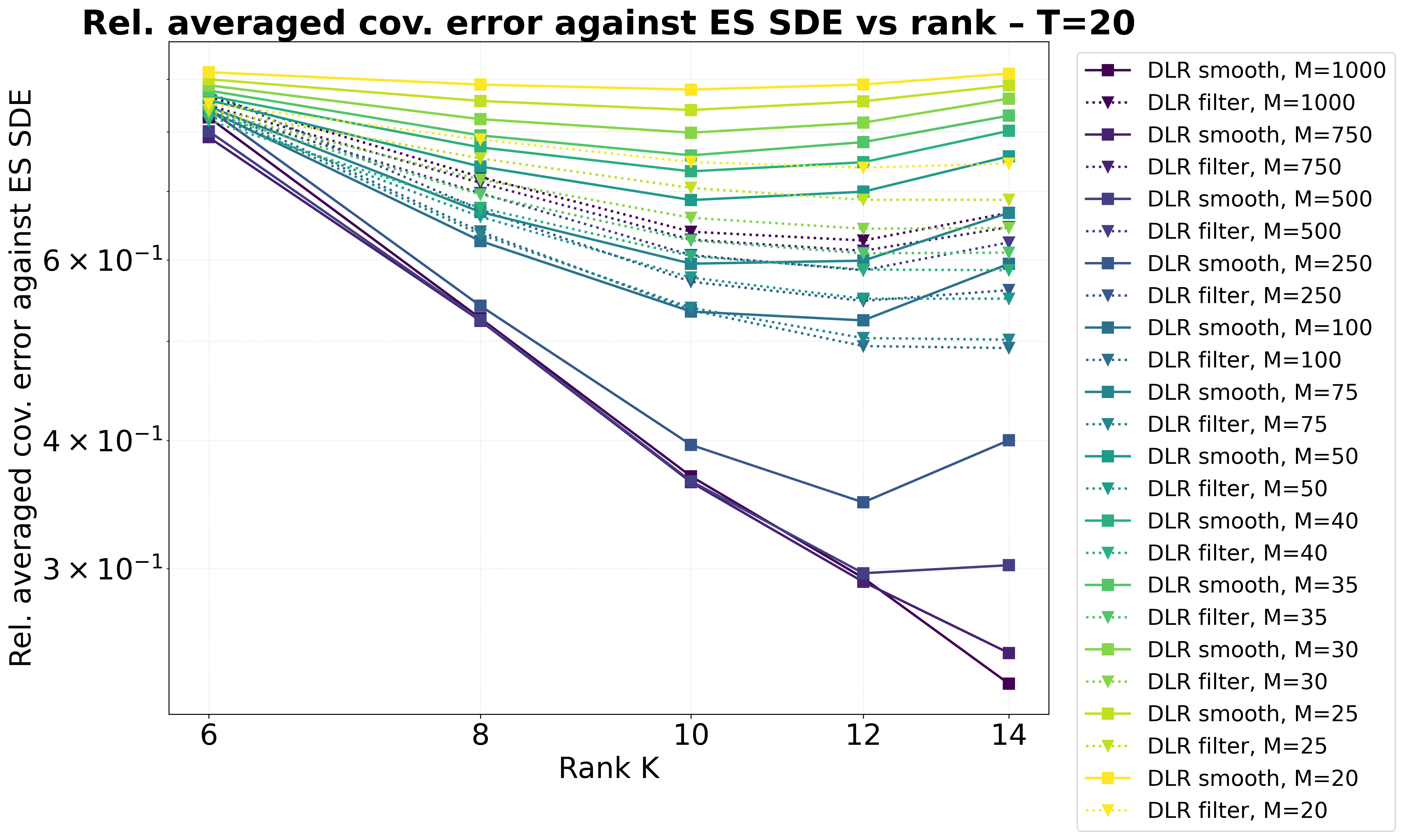} 
	\caption{Rel.\  averaged-over-time \emph{errors against ES SDE}, for continuous-time observations, with respect to the mean (left) and the covariance (right) for various rank $k$,  $d=50$, for problem \eqref{ex: SADR DA} with observation process \eqref{eq: SADR obs process 1}.}
	\label{fig: SADR CDCO ES SDE rank}
\end{figure}

\begin{figure}[!h]
	\centering
	\includegraphics[scale=0.18]{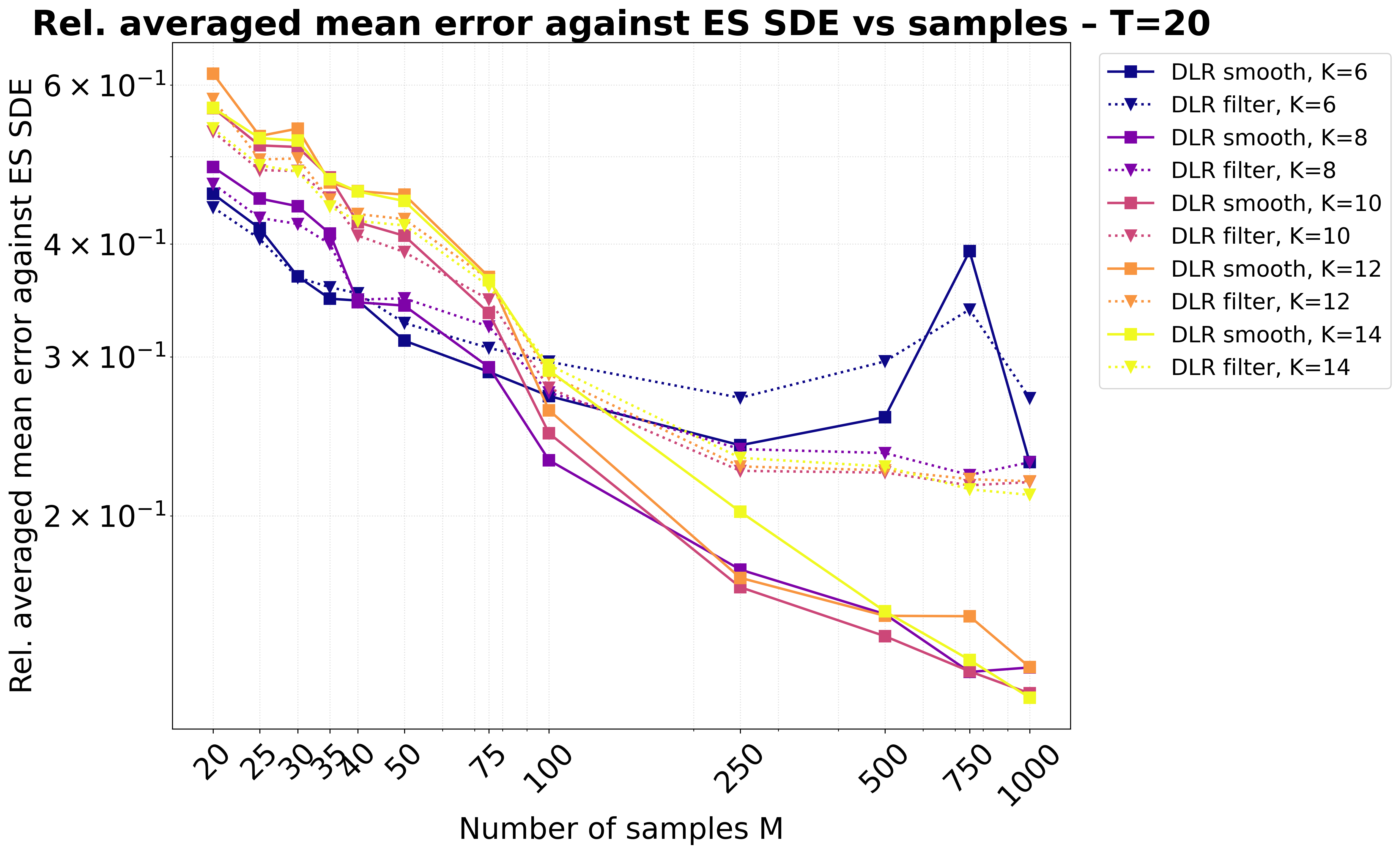}
	\includegraphics[scale=0.18]{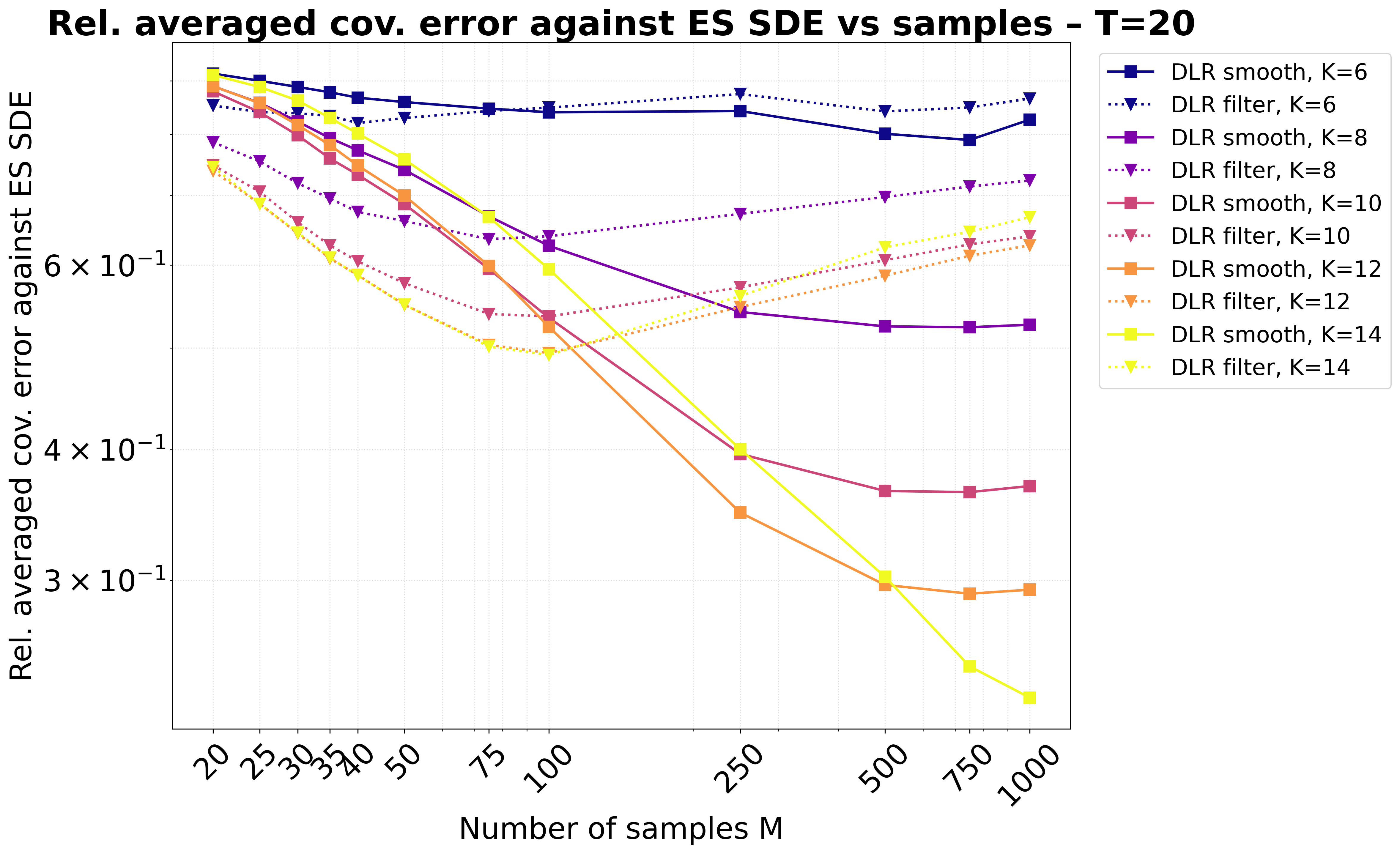} 
	\caption{Rel.\  averaged-over-time \emph{errors against ES SDE}, for continuous-time observations, with respect to the mean (left) and the covariance (right) for various number of samples $M$, for problem \eqref{ex: SADR DA} with observation process \eqref{eq: SADR obs process 1}.}
	\label{fig: SADR CDCO ES SDE samples}
\end{figure}

\section{Conclusion}
In this article, we extended the framework for the Kalman-type DLRA-JMCO filters proposed in \cite{marzouk2026filter} to smoothing. Exploiting the RTS recursion and the structure of the DLRA components recovers a smoothing procedure which is based only on low-rank matrix/vector multiplications, enabling fast simulation and light storage, with promising good accuracy when dealing with problems whose intrinsic dimension is actually low. Furthermore, we proposed a Kalman-Bucy-type smoother which is applicable in case of affine drift.

A possible extension of this work is to derive particle-type smoothers from the DLRA particle filter proposed in \cite{marzouk2026filter}, which is a work in progress.

\printbibliography

\end{document}